\documentclass[preprint,12pt]{elsarticle}

\usepackage[T1]{fontenc}
\usepackage{amsmath,amssymb,amsthm}
\usepackage{hyperref}
\biboptions{numbers,sort&compress}
\hypersetup{colorlinks=true,linkcolor=black,citecolor=black,urlcolor=black,
  pdftitle={The size of the poset of compatible locally quasi-convex topologies on locally compact abelian groups},
  pdfauthor={Zhouxiang Huang and Dekui Peng}}
\journal{Journal of Pure and Applied Algebra}

\newcommand{\R}{\mathbb R}
\newcommand{\Q}{\mathbb Q}
\newcommand{\Z}{\mathbb Z}
\newcommand{\N}{\mathbb N}
\newcommand{\T}{\mathbb T}
\newcommand{\F}{\mathbb F}
\newcommand{\cont}{\mathfrak c}
\newcommand{\C}{\mathcal C}
\newcommand{\Pow}{\mathcal P}
\newcommand{\Hom}{\operatorname{Hom}}
\newcommand{\width}{\operatorname{width}}
\newcommand{\rk}{\operatorname{r}}
\newcommand{\norm}[1]{\left\|#1\right\|}
\newcommand{\pol}{^{\triangleright}}
\newcommand{\prepol}{^{\triangleleft}}

\theoremstyle{plain}
\newtheorem{maintheorem}{Theorem}

\newtheorem{theorem}{Theorem}[section]
\newtheorem{lemma}[theorem]{Lemma}
\newtheorem{proposition}[theorem]{Proposition}
\newtheorem{corollary}[theorem]{Corollary}
\theoremstyle{definition}
\newtheorem{definition}[theorem]{Definition}
\theoremstyle{remark}
\newtheorem{remark}[theorem]{Remark}
\numberwithin{equation}{section}

\begin{document}
\begin{frontmatter}

\title{The size of the poset of compatible locally quasi-convex topologies on locally compact abelian groups}

\author[njnu]{Zhouxiang Huang}
\ead{2715383574@qq.com}

\address[njnu]{Institute of Mathematics, Nanjing Normal University,
Nanjing 210024, China}

\begin{abstract}
For a Hausdorff locally quasi-convex abelian group $G$, let
$\C(G)$ be the poset of all Hausdorff locally quasi-convex group
topologies on its underlying group having the same continuous
characters as $G$. We prove that, whenever $G$ is non-precompact,
$\C(G)$ contains an order-isomorphic copy of
$(\Pow(\cont),\subseteq)$, where $\cont=2^{\aleph_0}$. The embedding
takes values between the Bohr topology and the original topology.
Consequently, every infinite discrete abelian group $D$ satisfies
$|\C(D)|=\width\C(D)=2^{2^{|D|}}$. The discrete reduction for locally
compact abelian groups then yields exact cardinality and width
formulas for all such groups; in particular, both invariants equal
$2^{\cont}$ for every non-compact $\sigma$-compact locally compact
abelian group. These results answer Questions 6.1--6.3 and 6.5--6.7,
and the locally compact case of Question 6.4, posed by
L.~Au{\ss}enhofer and D.~Dikranjan in
\cite{AD20}. The embedding also applies to non-compact complete
metrizable locally quasi-convex groups. A non-compact precompact
nuclear group with a unique compatible topology shows that
non-compactness alone does not suffice in the nuclear setting.
Finally, the compatible poset of $\R^{\N}$ is not order-isomorphic
to that of any discrete abelian group.
\end{abstract}


\end{frontmatter}

\section{Introduction}\label{sec:introduction}

Two group topologies on an abelian group are called compatible if
they have the same continuous characters. The study of compatible
locally quasi-convex topologies is a group-theoretic counterpart of
Mackey theory for locally convex spaces. In this setting, the
continuous dual does not usually determine the topology uniquely.
It is therefore natural to study both the number of compatible
topologies and their order structure.

For a Hausdorff locally quasi-convex abelian group $(G,\tau)$, write
$\widehat G=(G,\tau)^\wedge$ and let $\C(G,\tau)$, or simply $\C(G)$,
denote the set of all Hausdorff locally quasi-convex group topologies
$\nu$ on $G$ satisfying $(G,\nu)^\wedge=\widehat G$, ordered by
inclusion. Its least element is the Bohr topology
$\sigma(G,\widehat G)$. For a poset $P$, its width is
\[
  \width P=\sup\{|A|:A\subseteq P\text{ is an antichain}\}.
\]
Throughout the paper, $\cont=2^{\aleph_0}$ and $\Pow(X)$ denotes
the power set of a set $X$.

Au{\ss}enhofer, Dikranjan, and Mart\'in-Peinador
\cite{ADMP15} obtained maximal cardinality and width for discrete
abelian groups of infinite rank. Au{\ss}enhofer and Dikranjan
\cite{AD20} subsequently showed that the compatible poset of every
locally compact abelian group is order-isomorphic to the compatible
poset of a suitable discrete group. They also constructed large
antichains and chains in the remaining finite-rank cases. Their
Questions 6.1--6.8 ask for exact cardinalities, behavior under finite
products, and extensions beyond the locally compact setting.

We use the $\mathcal S$-topology framework of
\cite[Section 3]{AD20}: the auxiliary topologies are topologies of
uniform convergence on selected subsets of the continuous dual.
Our separation argument also abstracts the use of differences of
selected points in \cite[Propositions 4.14 and 5.5]{AD20}.
To apply this approach to arbitrary non-precompact locally
quasi-convex groups, we use an infinite equicontinuous polar to
produce characters $(f_n)$ and points $(x_n)$ with uniformly small
off-diagonal evaluations and large diagonal evaluations. Uniform
convergence on almost disjoint subfamilies of these characters then
encodes every subset of a set of cardinality $\cont$.

The combinatorial part of this construction follows the framework
of \cite[Theorem 6.9]{ADMP15}, where uniform convergence on selected
coordinates of $c_0(X)$ is combined with an almost disjoint family
to embed $\Pow(\cont)$ into the compatible poset. Our contribution
relative to that argument is the construction of approximate
character coordinates from non-precompactness alone. The resulting
matrix estimates, together with differences of selected points,
replace the coordinate separation available in $c_0(X)$ and also
handle the finite Bohr constraints in the neighborhood test.
For uncountable discrete groups, the exact cardinality formula
below still uses the infinite-rank theorem of
\cite[Theorem 4.5]{ADMP15}.

\begin{maintheorem}\label{main:embedding}
Let $(G,\tau)$ be a non-precompact Hausdorff locally quasi-convex
abelian group. There is an order embedding
\[
  (\Pow(\cont),\subseteq)\longrightarrow\C(G,\tau)
\]
whose image lies in the interval
$[\sigma(G,\widehat G),\tau]$.
\end{maintheorem}

In particular, $\width\C(G)\ge2^{\cont}$, and $\C(G)$ contains a
chain of cardinality $\cont^+$, using
Lemma~\ref{lem:powerset-chain}. The theorem also gives a dichotomy
for arbitrary Hausdorff locally quasi-convex abelian groups:
either $\C(G)$ is a singleton or $\width\C(G)\ge2^{\cont}$.
In particular, $\C(G)$ is a chain if and only if it is a singleton.
The hypothesis of non-precompactness is essential, as the nuclear
example in Proposition~\ref{prop:nuclear-singleton} shows.

Combining Theorem~\ref{main:embedding} with the infinite-rank theorem
of \cite{ADMP15} and the discrete reduction of \cite{AD20} gives
the following formulas.

\begin{maintheorem}\label{main:cardinalities}
Every infinite discrete abelian group $D$ satisfies
\[
  |\C(D)|=\width\C(D)=2^{2^{|D|}}.
\]
Let $G=\R^n\times L$ be a locally compact abelian group, where
$L$ has a compact open subgroup $K$, and put
$D=\Z^n\times(L/K)$. Then
\[
  |\C(G)|=\width\C(G)=
  \begin{cases}
    1,&G\text{ is compact},\\
    2^{2^{|D|}},&G\text{ is non-compact}.
  \end{cases}
\]
For non-compact $\sigma$-compact $G$, both invariants equal
$2^{\cont}$.
\end{maintheorem}

Section~\ref{sec:embedding} proves Theorem~\ref{main:embedding}.
Section~\ref{sec:cardinalities} proves
Theorem~\ref{main:cardinalities}, and
Section~\ref{sec:questions} records its consequences for the
questions in \cite{AD20}. Section~\ref{sec:extensions} treats complete
metrizable groups and gives two counterexamples to broader
extensions. The general locally quasi-convex case of Question 6.4
and the full range of extensions requested in Question 6.8 are not
resolved here. All arguments are in ZFC.

\section{Preliminaries}\label{sec:preliminaries}

All groups are abelian and written additively. Group topologies are
Hausdorff unless an auxiliary topology is explicitly allowed to be
non-Hausdorff. We put
\[
  \T=\R/\Z,\qquad
  \norm{t}=\min_{m\in\Z}|s-m|\quad(t=s+\Z),\qquad
  \T_+=\{t\in\T:\norm{t}\le1/4\},
\]
and use $\N=\{1,2,\ldots\}$.

For a topological group $(G,\tau)$, a character is a continuous
homomorphism from $G$ to $\T$. The equality of dual groups in the
definition of compatibility is an equality of groups of functions
on the same underlying group. No equality of the compact-open
topologies on the dual groups is required. For $U\subseteq G$ and
$E\subseteq\widehat G$, define
\[
\begin{aligned}
  U\pol&=\{\chi\in\widehat G:\chi(U)\subseteq\T_+\},\\
  E\prepol&=\{x\in G:\chi(x)\in\T_+\text{ for every }\chi\in E\}.
\end{aligned}
\]
A set $U$ is quasi-convex if $U=(U\pol)\prepol$. The group is locally
quasi-convex if it has a neighborhood base at zero consisting of
quasi-convex sets. Such a Hausdorff group has enough characters to
separate its points. For a discrete abelian group $D$, $\C(D)$
always refers to compatibility with the discrete topology.

A group is precompact if finitely many translates of each
zero-neighborhood cover the group. The Bohr topology
$\sigma=\sigma(G,\widehat G)$ is the initial topology determined by
the characters in $\widehat G$. It is Hausdorff, precompact, locally
quasi-convex, and compatible with $\tau$. It is the least member of
$\C(G,\tau)$ and its unique precompact member; see
\cite{ComfortRoss,ADMP15}. Consequently, for every group topology
$\nu$,
\begin{equation}\label{eq:sandwich}
  \sigma\le\nu\le\tau
  \quad\Longrightarrow\quad
  (G,\nu)^\wedge=\widehat G.
\end{equation}
Indeed, the two inclusions of topologies give opposite inclusions
of the corresponding character groups. If $\nu$ is also locally
quasi-convex, it belongs to $\C(G,\tau)$.

The join of a family of group topologies is the least group
topology finer than every member of the family. Its
zero-neighborhood base consists of finite intersections of
zero-neighborhoods from the given topologies. In particular, joins
of locally quasi-convex topologies are locally quasi-convex whenever
the characters used to separate their basic quasi-convex
neighborhoods are continuous in the join.

For an abelian group $H$ and an integer $j\ge1$, let
$j_H:H\to H$ be the map $h\mapsto jh$. Thus, for $B\subseteq H$,
$j_H^{-1}(B)=\{h\in H:jh\in B\}$ is the inverse image of $B$
under this map. We shall use the small-arc identity
\begin{equation}\label{eq:small-arcs}
  \bigcap_{j=1}^m j_{\T}^{-1}(\T_+)
  =\{t\in\T:\norm{t}\le1/(4m)\},\qquad m\ge1.
\end{equation}
The left-hand side consists of those $t$ for which $jt\in\T_+$
for every $1\le j\le m$. One inclusion follows from
$\norm{jt}\le j\norm{t}$. The other follows by induction: if
$\norm{t}\le1/(4(m-1))$ and $mt\in\T_+$, then the representative
of $mt$ obtained by multiplying the representative of $t$ has
absolute value at most $m/(4(m-1))\le1/2$, and hence at most $1/4$.
The case $m=1$ is immediate.

\subsection{Uniform convergence on countable sets of characters}
\label{subsec:uniform-topologies}

In the proof of Theorem~\ref{main:embedding} and in
\ref{app:sparse}, we use topologies of uniform convergence on
countable subsets of $\widehat G$. We collect the properties
needed in these constructions, following
\cite[Lemmas 3.1 and 3.2 and Definition 3.5]{AD20}.
The statements hold for arbitrary subsets of $\widehat G$.

For $S\subseteq\widehat G$, let $\tau_S$ be the topology of
uniform convergence on $S$, with zero-neighborhood base
\begin{equation}\label{eq:uniform-neighborhoods}
  V(S,k)=\{x\in G:\norm{\chi(x)}\le1/(4k)
    \text{ for every }\chi\in S\},\qquad k\ge1.
\end{equation}
For later estimates, put
\[
  q_S(x)=\sup_{\chi\in S}\norm{\chi(x)},\qquad q_\varnothing=0,
\]
so that $V(S,k)=\{x:q_S(x)\le1/(4k)\}$.
The ambient group will always be clear, so it is omitted from
the notation. These auxiliary topologies may be non-Hausdorff;
in particular, $\tau_\varnothing$ is indiscrete.

\begin{proposition}\label{prop:uniform-basic}
Let $(G,\tau)$ be a topological abelian group and let
$S,T\subseteq\widehat G$.
\begin{enumerate}
\item[(a)] The function $q_S$ is symmetric and subadditive.
The topology $\tau_S$ is a group topology with a quasi-convex
zero-neighborhood base. Every $\chi\in S$ is
$\tau_S$-continuous, and
\begin{equation}\label{eq:uniform-polar}
  V(S,k)=\left(\bigcup_{j=1}^k jS\right)\prepol,
  \qquad jS=\{j\chi:\chi\in S\}.
\end{equation}
\item[(b)] If $S$ is $\tau$-equicontinuous, then $\tau_S\le\tau$.
\item[(c)] We have
\[
  q_{S\cup T}=\max\{q_S,q_T\},\qquad
  \tau_{S\cup T}=\tau_S\vee\tau_T.
\]
The join identity extends to finite unions.
\item[(d)] If $S$ is finite, $\tau_S$ is the initial topology
determined by the characters in $S$.
\end{enumerate}
\end{proposition}

\begin{proof}
(a) Symmetry and subadditivity of $q_S$ follow from the norm on
$\T$. They give
\[
  -V(S,k)=V(S,k),\qquad
  V(S,2k)+V(S,2k)\subseteq V(S,k).
\]
Thus the decreasing family $\{V(S,k):k\ge1\}$ defines a group
topology. Each $\chi\in S$ is continuous because
$\norm{\chi(x)}\le q_S(x)$. By \eqref{eq:small-arcs},
$V(S,k)$ is the inverse polar in \eqref{eq:uniform-polar},
hence quasi-convex.

(b) Equicontinuity makes every $V(S,k)$ a
$\tau$-zero-neighborhood, so $\tau_S\le\tau$.

(c) The maximum identity follows from the definition of $q_S$.
It gives $V(S\cup T,k)=V(S,k)\cap V(T,k)$, proving the join
identity.

(d) For finite $S$, the sets
$V(S,k)=\bigcap_{\chi\in S}\{x:\norm{\chi(x)}\le1/(4k)\}$
form a neighborhood base for the initial topology.
\end{proof}

For a family $\mathcal S\subseteq\Pow(\widehat G)$, write
\[
  \tau_{\mathcal S}=\bigvee_{S\in\mathcal S}\tau_S,
\]
as in \cite[Definition 3.5]{AD20}. An empty join is indiscrete.
Thus $S$ indexes one character set, whereas $\mathcal S$
indexes a family of such sets.

\begin{proposition}\label{prop:uniform-joins}
Let $(G,\tau)$ be a Hausdorff locally quasi-convex abelian group,
put $\sigma=\sigma(G,\widehat G)$, and let
$\mathcal S\subseteq\Pow(\widehat G)$.
\begin{enumerate}
\item[(a)] The topology $\tau_{\mathcal S}$ has a quasi-convex
zero-neighborhood base
\[
  V(S_1\cup\cdots\cup S_r,k),\qquad
  S_1,\ldots,S_r\in\mathcal S,\quad k\ge1,
\]
where $r$ is finite and $r=0$ gives $G$.
\item[(b)] The Bohr topology is
\[
  \sigma=\bigvee_{\substack{F\subseteq\widehat G\\F\text{ finite}}}
    \tau_F.
\]
Thus adjoining all finite subsets of $\widehat G$ to
$\mathcal S$ gives $\sigma\vee\tau_{\mathcal S}$.
\item[(c)] If every $S\in\mathcal S$ is $\tau$-equicontinuous,
then
\[
  \nu=\sigma\vee\tau_{\mathcal S}\in\C(G,\tau).
\]
\end{enumerate}
\end{proposition}

\begin{proof}
(a) A basic neighborhood in the join is a finite intersection
of sets $V(S_i,k_i)$. For $r\ge1$, taking $k=\max_i k_i$ gives
the smaller quasi-convex neighborhood
\[
  \bigcap_{i=1}^r V(S_i,k)=V(S_1\cup\cdots\cup S_r,k).
\]
The empty intersection gives $G$.

(b) This follows from Proposition~\ref{prop:uniform-basic}(d)
and the definition of $\sigma$.

(c) Proposition~\ref{prop:uniform-basic}(b) gives
$\sigma\le\nu\le\tau$. By (a)--(b), $\nu$ is locally
quasi-convex. It is Hausdorff because $\sigma$ is Hausdorff
and $\sigma\le\nu$; \eqref{eq:sandwich} gives compatibility.
\end{proof}

\begin{remark}
The finite-union identity in
Proposition~\ref{prop:uniform-basic}(c) need not hold for infinite unions.
For example, let $G=\T^\N$ have its product topology and put
$S=\{\pi_n:n\in\N\}$, where $\pi_n$ is the $n$th coordinate
character. Then $\bigvee_n\tau_{\{\pi_n\}}$ is the product
topology, whereas $\tau_S$ is strictly finer: $V(S,1)$ contains
no product zero-neighborhood. Also, $\tau_S$ alone need not be
compatible with $\tau$; Proposition~\ref{prop:uniform-joins}(c)
is the compatibility statement used below.
\end{remark}

\subsection{Equicontinuous families and compact polars}

\begin{lemma}\label{lem:finite-equicont}
Every equicontinuous family of characters of a precompact abelian
group is finite. Thus, if $P$ is precompact and
$M\subseteq\widehat P$ is infinite, then for every
zero-neighborhood $V$ in $P$ there exist $f\in M$ and $x\in V$
such that $\norm{f(x)}>1/4$.
\end{lemma}

\begin{proof}
Let $M$ be equicontinuous. Given $\varepsilon>0$, choose a
zero-neighborhood $V$ such that
$\norm{f(v)}<\varepsilon/4$ for all $f\in M$ and $v\in V$.
Precompactness gives a finite set of centers
$x_1,\ldots,x_r$ with $P=\bigcup_{i=1}^r(x_i+V)$.
Equip $M$ with the uniform metric
\[
  d_\infty(f,g)=\sup_{x\in P}\norm{f(x)-g(x)}.
\]
The evaluation vectors $(f(x_i))_{i=1}^r$, $f\in M$, lie in the
compact metric space $\T^r$, whose metric is
$d_r(z,w)=\max_{1\le i\le r}\norm{z_i-w_i}$.
Consequently, there are finitely many characters
$g_1,\ldots,g_s\in M$ such that for every $f\in M$ there is a $k$
with
\[
  \max_{1\le i\le r}\norm{f(x_i)-g_k(x_i)}<\varepsilon/4.
\]
Fix such $f$ and $k$. Every $x\in P$ has a representation
$x=x_i+v$ with $v\in V$, and hence
\[
  \norm{f(x)-g_k(x)}
  \le\norm{f(x_i)-g_k(x_i)}+\norm{f(v)}+\norm{g_k(v)}
  <3\varepsilon/4.
\]
Thus $d_\infty(f,g_k)\le3\varepsilon/4<\varepsilon$.
The finite family $\{g_1,\ldots,g_s\}$ is therefore an
$\varepsilon$-net for $M$ in the uniform metric. Since
$\varepsilon>0$ was arbitrary, $M$ is totally bounded.

If $f\ne g$, the image of $f-g$ is a nonzero subgroup of $\T$.
Every such subgroup contains an element of norm at least $1/3$.
Indeed, choose a nonzero element, change its sign if necessary,
and represent it by $a\in(0,1/2]$. If $a<1/3$, the least positive
integer $n$ with $na\ge1/3$ satisfies $1/3\le na<2/3$, so
$\norm{na+\Z}\ge1/3$; if $a\ge1/3$, the chosen element already
has the required norm. Consequently,
\[
  f\ne g\quad\Longrightarrow\quad d_\infty(f,g)\ge1/3.
\]
Now take a finite $1/4$-net $\{g_1,\ldots,g_s\}\subseteq M$ as
above. For each $f\in M$, the inequality
$d_\infty(f,g_k)<1/4$ for some $k$ forces $f=g_k$.
Thus $M=\{g_1,\ldots,g_s\}$ is finite.

For the second assertion, suppose instead that
$f(V)\subseteq\T_+$ for every $f\in M$. The zero-neighborhoods
$V_m=\bigcap_{j=1}^m j_P^{-1}(V)$ then satisfy
$\norm{f(x)}\le1/(4m)$ for all $f\in M$ and $x\in V_m$, by
\eqref{eq:small-arcs}. This would make $M$ equicontinuous and
hence finite, a contradiction.
\end{proof}

\begin{lemma}\label{lem:infinite-polar}
Every non-precompact locally quasi-convex group $(G,\tau)$ has a
quasi-convex zero-neighborhood $U$ whose polar $E=U\pol$ is
infinite, equicontinuous, and compact for pointwise convergence
on $G$.
\end{lemma}

\begin{proof}
Choose a zero-neighborhood having no finite covering by translates,
and choose a quasi-convex zero-neighborhood $U$ contained in it.
No finite collection of translates of $U$ covers $G$.

Work in the compact product $\T^G$, whose induced topology on a
family of functions is the topology of pointwise convergence.
For $x\in G$, the evaluation map $e_x:\T^G\to\T$, given by
$e_x(p)=p(x)$, is continuous. The set of abstract homomorphisms is
\[
  \Hom(G,\T)=\bigcap_{x,y\in G}
  \{p\in\T^G:p(x+y)-p(x)-p(y)=0\}.
\]
Each set in this intersection is closed, since
$p\mapsto p(x+y)-p(x)-p(y)$ is continuous and $\{0\}$ is closed
in $\T$. Thus $\Hom(G,\T)$ is closed in $\T^G$.
Now put
\[
  K_U=\Hom(G,\T)\cap\bigcap_{u\in U}e_u^{-1}(\T_+).
\]
Since $\T_+$ is closed in $\T$, each $e_u^{-1}(\T_+)$ is closed.
Consequently, $K_U$ is closed in $\T^G$ and hence compact.

It remains to identify $K_U$ with $E=U\pol$. For every $m\ge1$,
the set $U_m=\bigcap_{j=1}^mj_G^{-1}(U)$ is a
$\tau$-zero-neighborhood. If $h\in K_U$ and $x\in U_m$, then
$jh(x)=h(jx)\in\T_+$ for $1\le j\le m$, so
\eqref{eq:small-arcs} gives $\norm{h(x)}\le1/(4m)$.
These estimates prove that every $h\in K_U$ is $\tau$-continuous,
and they hold uniformly for all $h\in K_U$. Thus
$K_U\subseteq\widehat G$, $K_U=E$, and $E$ is equicontinuous.
The equality $E=K_U$ also proves that $E$ is closed in $\T^G$
and compact for pointwise convergence.

If $E$ were finite, quasi-convexity would imply $U=E\prepol$.
This set contains a Bohr zero-neighborhood determined by the
finitely many characters in $E$. Precompactness of the Bohr
topology would then give a finite covering of $G$ by translates
of $U$, a contradiction.
\end{proof}

\subsection{Chains in power sets}
\label{subsec:powerset-chains}

The following standard set-theoretic fact is used in
\cite[proof of Theorem A, p.~336]{AD20}. We include a proof to
make the chain consequence of Theorem~\ref{main:embedding}
self-contained.

\begin{lemma}\label{lem:powerset-chain}
For every infinite cardinal $\kappa$, the poset
$(\Pow(\kappa),\subseteq)$ contains a chain of cardinality
$\kappa^+$. In particular, $(\Pow(\cont),\subseteq)$ contains
a chain of cardinality $\cont^+$.
\end{lemma}

\begin{proof}
Regard cardinals as initial ordinals. Let $\mu$ be the least
infinite cardinal such that $2^\mu>\kappa$. Cantor's theorem
gives $\mu\le\kappa$. Put
\[
  T=\bigcup_{\alpha<\mu}\{0,1\}^{\alpha},\qquad
  B=\{0,1\}^{\mu}.
\]
Thus $T$ consists of the binary sequences of length less than
$\mu$, and $B$ consists of the branches of length $\mu$.
For every $\alpha<\mu$, we have $2^{|\alpha|}\le\kappa$:
this follows from the minimality of $\mu$ when $\alpha$ is
infinite, and from the infinitude of $\kappa$ when $\alpha$
is finite. Consequently,
\[
  |T|\le\sum_{\alpha<\mu}2^{|\alpha|}
       \le\mu\cdot\kappa=\kappa.
\]

Order $T\cup B$ lexicographically. When a node in $T$ ends,
use a terminal symbol $*$ with $0<*<1$; compare distinct
sequences at the first position where their symbols differ.
For each $b\in B$, define the lower cut
\[
  L_b=\{s\in T:s<b\}.
\]
If $b,c\in B$ and $b<c$, let
\[
  \alpha=\min\{\beta<\mu:b(\beta)\ne c(\beta)\}.
\]
Then $b(\alpha)=0$, $c(\alpha)=1$, and their common prefix
$s=b\mathbin{\upharpoonright}\alpha
  =c\mathbin{\upharpoonright}\alpha\in T$
satisfies $b<s<c$. Hence $L_b\subseteq L_c$ and
$s\in L_c\setminus L_b$, so $L_b\subsetneq L_c$.
It follows that
\[
  \{L_b:b\in B\}\subseteq\Pow(T)
\]
is a chain of cardinality $|B|=2^\mu>\kappa$.
Choose an injection $i:T\to\kappa$. The images $i[L_b]$
form a chain in $\Pow(\kappa)$ of the same cardinality,
which is at least $\kappa^+$. Selecting $\kappa^+$ members
gives the desired chain. Taking $\kappa=\cont$ proves the
particular case.
\end{proof}

Here, a chain is a family linearly ordered by inclusion;
$\kappa^+$ specifies its cardinality, not its order type.

\section{The order embedding}\label{sec:embedding}

\begin{proof}[Proof of Theorem~\ref{main:embedding}]
Write $\sigma=\sigma(G,\widehat G)$ and take the infinite compact
equicontinuous set $E$ from Lemma~\ref{lem:infinite-polar}.
An infinite compact Hausdorff space has a non-isolated point;
choose $\chi\in E$ with this property. Put
$\varepsilon_n=2^{-n-6}$ for $n\ge1$.

We first construct $f_n\in E-\chi$ and $x_n\in G$ satisfying
\begin{equation}\label{eq:matrix}
  \norm{f_n(x_n)}>1/4,\qquad
  \norm{f_j(x_n)}<\varepsilon_n\quad(j\ne n).
\end{equation}
Suppose $f_i,x_i$ have been chosen for $i<n$. The set
\[
  M_n=\{\psi-\chi:\psi\in E,
    \ \norm{(\psi-\chi)(x_i)}<\varepsilon_i\text{ for }i<n\}
\]
is infinite: the conditions on $\psi$ define a relative open
neighborhood of the non-isolated point $\chi$ in the Hausdorff
space $E$, and every such neighborhood is infinite. Apply
Lemma~\ref{lem:finite-equicont} to the precompact group
$(G,\sigma)$, the infinite family $M_n$, and the neighborhood
\[
  V_n=\{x\in G:\norm{f_i(x)}<\varepsilon_n\text{ for }i<n\}.
\]
Every member of $M_n$ is a $\sigma$-continuous character. The lemma
gives $f_n\in M_n$ and $x_n\in V_n$ with
$\norm{f_n(x_n)}>1/4$. The definition of $M_n$ controls the new
character at the earlier points, while the definition of $V_n$
controls the earlier characters at the new point. After completing
the recursion, \eqref{eq:matrix} holds for all indices.

For $A\subseteq\N$, put $S_A=\{f_n:n\in A\}$. Let $\tau_A$
be the topology on $G$ of uniform convergence with respect to
the character family $S_A$.
By Proposition~\ref{prop:uniform-basic}, it has the quasi-convex
zero-neighborhood base
\[
  V(A,k)=\{x\in G:\norm{f_j(x)}\le1/(4k)
    \text{ for every }j\in A\},\qquad k\ge1.
\]
The topology $\tau_A$ need not be Hausdorff. Since $E-\chi$ is
$\tau$-equicontinuous, part (b) of that proposition gives
$\tau_A\le\tau$.

Fix an almost disjoint family
$\{A_\xi:\xi<\cont\}\subseteq[\N]^\omega$; thus each $A_\xi$
is infinite and $A_\xi\cap A_\zeta$ is finite for $\xi\ne\zeta$.
Such a family is obtained from the sets of initial segments of
the infinite branches of the countable binary tree. For
$J\subseteq\cont$, define
\begin{equation}\label{eq:nuJ}
  \nu_J=\sigma\vee\bigvee_{\xi\in J}\tau_{A_\xi}.
\end{equation}
In the notation for $\mathcal S$-topologies, this is
$\nu_J=\tau_{\mathcal S_J}$, where
\[
  \mathcal S_J=[\widehat G]^{<\omega}
      \cup\{S_{A_\xi}:\xi\in J\},
\]
and $[\widehat G]^{<\omega}$ is the family of all finite subsets
of $\widehat G$. Proposition~\ref{prop:uniform-joins}(b)--(c)
gives $\sigma\le\nu_J\le\tau$ and
$\nu_J\in\C(G,\tau)$.

For each $\xi<\cont$, let $W_\xi=V(A_\xi,2)$. We claim that
\begin{equation}\label{eq:membership-test}
  W_\xi\text{ is a }\nu_J\text{-zero-neighborhood}
  \quad\Longleftrightarrow\quad\xi\in J.
\end{equation}
If $\xi\in J$, this is immediate from \eqref{eq:nuJ}.
Suppose $\xi\notin J$. A $\nu_J$-zero-neighborhood base consists
of the sets
\[
  B(A,F,\eta)=\{x\in G:\norm{\phi(x)}\le\eta
    \text{ for every }\phi\in S_A\cup F\},
\]
where $F\subseteq\widehat G$ is finite, $\eta>0$, and
$A=\bigcup_{\zeta\in J_0}A_\zeta$ for a finite $J_0\subseteq J$.
Indeed, Proposition~\ref{prop:uniform-joins}(a) allows us to
combine the finitely many character sets involved and use a
smaller common bound. Empty $F$ and $J_0$ are allowed.

Almost disjointness implies that $A_\xi\setminus A$ is infinite.
Restrict to an infinite tail on which $\varepsilon_n<\eta/2$.
Compactness gives a finite cover of $\T^F$ by sets on each of
which every coordinate varies by less than $\eta$. Since the
tail is infinite, there are distinct indices $n,m$ whose
evaluation vectors $(\psi(x_n))_{\psi\in F}$ and
$(\psi(x_m))_{\psi\in F}$ belong to the same member of this cover.
Hence
\[
  \norm{\psi(x_n)-\psi(x_m)}<\eta\qquad(\psi\in F).
\]
For every $j\in A$, we have $j\ne n,m$, so \eqref{eq:matrix} gives
\[
  \norm{f_j(x_n-x_m)}
  \le\norm{f_j(x_n)}+\norm{f_j(x_m)}
  <\varepsilon_n+\varepsilon_m<\eta.
\]
Thus $x_n-x_m\in B(A,F,\eta)$. On the other hand,
\[
  \norm{f_n(x_n-x_m)}
  \ge\norm{f_n(x_n)}-\norm{f_n(x_m)}
  >1/4-\varepsilon_m>1/8.
\]
Since $n\in A_\xi$, this gives $x_n-x_m\notin W_\xi$.
Thus no basic $\nu_J$-zero-neighborhood is contained in $W_\xi$,
proving \eqref{eq:membership-test}.

The assignment $J\mapsto\nu_J$ is order-preserving. If
$\nu_J\le\nu_{J'}$, every $\nu_J$-neighborhood is also a
$\nu_{J'}$-neighborhood. Applying \eqref{eq:membership-test} to
$W_\xi$ for each $\xi\in J$ gives $J\subseteq J'$. Consequently,
\[
  \nu_J\le\nu_{J'}\quad\Longleftrightarrow\quad J\subseteq J'.
\]
If $J\subsetneq J'$, choose $\xi\in J'\setminus J$.
Then $W_\xi$ is a $\nu_{J'}$-zero-neighborhood but not a
$\nu_J$-zero-neighborhood, so $\nu_J<\nu_{J'}$.
Thus strict inclusion, equality, and incomparability of indexing
sets correspond exactly to the respective relations between the
associated topologies. The assignment is an order embedding.
\end{proof}

\begin{corollary}\label{cor:width}
Every non-precompact Hausdorff locally quasi-convex abelian group
$G$ satisfies
\[
  |\C(G)|\ge\width\C(G)\ge2^{\cont},
\]
and $\C(G)$ contains a chain of cardinality $\cont^+$.
\end{corollary}

\begin{proof}
Partition a set of cardinality $\cont$ into two-element blocks.
The subsets choosing exactly one element from each block form
an antichain in its power set of cardinality $2^{\cont}$.
Theorem~\ref{main:embedding} transfers this antichain to $\C(G)$.
For the chain assertion, Lemma~\ref{lem:powerset-chain} gives
a chain of cardinality $\cont^+$ in $\Pow(\cont)$, whose image
under the order embedding in Theorem~\ref{main:embedding}
is a chain of the same cardinality in $\C(G)$.
\end{proof}

\begin{corollary}\label{cor:universality}
If $G$ is a non-precompact Hausdorff locally quasi-convex abelian
group, every poset of cardinality at most $\cont$ order-embeds into
$[\sigma(G,\widehat G),\tau]\subseteq\C(G,\tau)$.
\end{corollary}

\begin{proof}
For a poset $P$, the map $p\mapsto\{q\in P:q\le p\}$ is an order
embedding into $\Pow(P)$. If $|P|\le\cont$, compose it with an
embedding into $\Pow(\cont)$ and with Theorem~\ref{main:embedding}.
\end{proof}

\begin{corollary}\label{cor:dichotomy}
For every Hausdorff locally quasi-convex abelian group $G$, either
$|\C(G)|=1$ or $\width\C(G)\ge2^{\cont}$.
Consequently, $\C(G)$ is a chain if and only if it is a singleton.
\end{corollary}

\begin{proof}
The Bohr topology is the unique precompact member of $\C(G)$.
If there is another member $\nu$, apply
Theorem~\ref{main:embedding} to $(G,\nu)$ and observe that
$\C(G,\nu)=\C(G)$.
\end{proof}

\begin{remark}
Corollary~\ref{cor:dichotomy} strengthens the conclusion proposed
in \cite[Conjecture 7.6]{ADMP15} and rules out the finite
non-singleton possibility asked about in
\cite[Question 7.5]{ADMP15}. It also answers the chain question
recorded in \cite[Question 1.1(b)]{AD20} for Hausdorff locally
quasi-convex abelian groups.
\end{remark}

\section{Exact cardinalities}\label{sec:cardinalities}

We use two results from the literature. First, the infinite-rank
theorem of \cite[Theorem 4.5]{ADMP15}, also stated as Theorem 1.2 of
\cite{AD20}, gives
\begin{equation}\label{eq:infinite-rank}
  |\C(D)|=\width\C(D)=2^{2^{|D|}}
\end{equation}
for every discrete abelian group $D$ of infinite rank. The rank
convention is
\[
  \rk(D)=r_0(D)+\sum_{p\text{ prime}}r_p(D),
\]
where $r_0(D)=\dim_{\Q}(\Q\otimes_{\Z}D)$ and
$r_p(D)=\dim_{\F_p}D[p]$, with $D[p]=\{x\in D:px=0\}$.
Second, Theorem C of \cite{AD20} states that if
$G=\R^n\times L$ is locally compact abelian and $K\le L$ is
compact and open, then
\begin{equation}\label{eq:discrete-reduction}
  \C(G)\cong\C\bigl(\Z^n\times(L/K)\bigr)
\end{equation}
as ordered sets. The standard structure theorem for locally
compact abelian groups ensures that every such $G$ admits this
form, with $n\ge0$.

\begin{theorem}\label{thm:discrete}
For every discrete abelian group $D$,
\[
  |\C(D)|=\width\C(D)=
  \begin{cases}
    1,&D\text{ is finite},\\
    2^{2^{|D|}},&D\text{ is infinite}.
  \end{cases}
\]
\end{theorem}

\begin{proof}
A finite group has only one Hausdorff topology. If $D$ is
countably infinite, its discrete topology is non-precompact.
Corollary~\ref{cor:width} gives an antichain of size $2^{\cont}$,
while the number of topologies on a countable set is at most
$|\Pow(\Pow(D))|=2^{\cont}$. This proves both equalities.

If $D$ is uncountable, then $\rk(D)$ is infinite. To see this,
suppose instead that the displayed rank is finite and put
$T=t(D)$, the torsion subgroup. For each prime $p$, the group
$T[p]$ is finite. Multiplication by $p$ maps $T[p^k]$ into
$T[p^{k-1}]$ with kernel $T[p]$, so induction shows that every
$T[p^k]$ is finite. Each primary component of $T$ is countable,
and there are only countably many primes; hence $T$ is countable.
The torsion-free quotient $D/T$ embeds in a finite-dimensional
rational vector space and is also countable. Thus $D$ is a
countable union of countable cosets and would be countable. The
infinite-rank formula \eqref{eq:infinite-rank} now applies.
\end{proof}

\begin{theorem}\label{thm:lca}
Let $G=\R^n\times L$ be locally compact abelian, let $K\le L$
be compact and open, and put $D=\Z^n\times(L/K)$. Then
\[
  |\C(G)|=\width\C(G)=
  \begin{cases}
    1,&G\text{ is compact},\\
    2^{2^{|D|}},&G\text{ is non-compact}.
  \end{cases}
\]
If $G$ is non-compact and $\sigma$-compact, then
$|\C(G)|=\width\C(G)=2^{\cont}$.
\end{theorem}

\begin{proof}
Use \eqref{eq:discrete-reduction} and
Theorem~\ref{thm:discrete}. The group $D$ is finite precisely
when $n=0$ and $L/K$ is finite, which is equivalent to compactness
of $G$. If $G$ is $\sigma$-compact, so is $L$, and its discrete
quotient $L/K$ is countable: every compact subset has finite
image in the quotient. Thus $D$ is countably infinite when
$G$ is also non-compact.
\end{proof}

Theorems~\ref{thm:discrete} and \ref{thm:lca} prove
Theorem~\ref{main:cardinalities}.

\section{Consequences for the questions of Au{\ss}enhofer and Dikranjan}
\label{sec:questions}

The exact formulas above settle the cardinality and width
questions in \cite[Section 6]{AD20}. We collect the consequences
to make their scope explicit.

\begin{corollary}\label{cor:basic-groups}
For every prime $p$ and every positive integer $m$,
\[
\begin{split}
  |\C(\Z^m)|=\width\C(\Z^m)&=2^{\cont},\\
  |\C(\Z(p^\infty)^m)|
    =\width\C(\Z(p^\infty)^m)&=2^{\cont}.
\end{split}
\]
\end{corollary}

\begin{proof}
Each group in the statement is countably infinite and discrete.
Apply Theorem~\ref{thm:discrete}.
\end{proof}

For $m=1$, Corollary~\ref{cor:basic-groups} answers Question 6.1:
the cardinalities for $\Z$ and all Pr\"ufer groups are equal to
$2^{\cont}$, independently of $p$. It also answers Question 6.7
by giving the corresponding width equalities. For arbitrary $m$,
it answers Question 6.5: both sequences of cardinalities are
constant. \ref{app:sparse} provides an explicit
sparse-character construction for the two basic groups.

Theorem~\ref{thm:lca} answers Question 6.2 in ZFC, since every
non-compact $\sigma$-compact locally compact abelian group satisfies
\[
  \cont<|\C(G)|=2^{\cont}.
\]
It also answers Question 6.6: cardinality and width agree for
every non-compact locally compact abelian group.

\begin{corollary}\label{cor:finite-products}
Let $G,G_1,G_2$ be locally compact abelian groups and let $m\ge1$
be an integer. Then
\[
  |\C(G^m)|=|\C(G)|,
  \qquad
  |\C(G_1\times G_2)|
    =\max\{|\C(G_1)|,|\C(G_2)|\}.
\]
\end{corollary}

\begin{proof}
Choose discrete reductions $D,D_1,D_2$ as in
\eqref{eq:discrete-reduction}. The same construction gives
\[
  \C(G^m)\cong\C(D^m),\qquad
  \C(G_1\times G_2)\cong\C(D_1\times D_2).
\]
For finite $D$, both $D$ and $D^m$ have singleton compatible
posets. For infinite $D$, $|D^m|=|D|$, so the first equality
follows from Theorem~\ref{thm:discrete}. If both $D_i$ are finite,
all three cardinalities in the second equality are $1$.
Otherwise, $|D_1\times D_2|=\max\{|D_1|,|D_2|\}$, and the
same theorem gives the second equality.
\end{proof}

This answers Question 6.3 and the locally compact case of
Question 6.4. The product equality for arbitrary locally
quasi-convex groups remains unresolved here: the lower bound in
Theorem~\ref{main:embedding} does not give the upper bound that
such an equality would require.

\section{Extensions and counterexamples}\label{sec:extensions}

Question 6.8 of \cite{AD20} asks which results extend to
non-compact complete metrizable locally quasi-convex groups or
to non-compact nuclear groups. We establish a general lower
bound in the first class and exhibit two obstructions to broader
extensions. Nuclear groups are understood in the sense of
Banaszczyk \cite{Banaszczyk}; we use the facts that precompact
abelian groups and locally compact abelian groups are nuclear,
and that the class of nuclear groups is closed under products
and subgroups \cite[Propositions 7.5, 7.6 and 7.10]{Banaszczyk}.
For precompact groups, this follows by viewing them as subgroups
of their compact completions.

\subsection{Complete metrizable groups}

\begin{theorem}\label{thm:complete}
Let $G$ be a non-compact complete metrizable locally quasi-convex
abelian group. Then $\Pow(\cont)$ order-embeds in $\C(G)$.
Consequently,
\[
  |\C(G)|\ge\width\C(G)\ge2^{\cont},
\]
and $\C(G)$ contains a chain of cardinality $\cont^+$.
If $G$ is also separable, then
$|\C(G)|=\width\C(G)=2^{\cont}$.
\end{theorem}

\begin{proof}
A complete precompact group is compact. Hence $G$ is
non-precompact, and Theorem~\ref{main:embedding} and
Corollary~\ref{cor:width} give the embedding, the lower bound,
and the chain.

By the Mackey theorem for complete metrizable locally
quasi-convex groups \cite[Theorem 4.2]{CMT99}, every compatible locally
quasi-convex topology is coarser than the original topology
$\tau$. If $G$ is separable and metrizable, $\tau$ has a
countable base and $|\tau|\le\cont$. Therefore
\[
  |\C(G)|\le|\Pow(\tau)|\le2^{\cont}.
\]
The antichain lower bound proves the final assertion.
\end{proof}

In particular, the cardinality and width bounds and the chain
conclusion of Theorem D of \cite{AD20} extend to non-compact
complete metrizable locally quasi-convex abelian groups, with
the stronger antichain bound $2^{\cont}$.

\subsection{A non-compact nuclear group with a unique compatible topology}

\begin{proposition}\label{prop:nuclear-singleton}
Give
\[
  E=\bigoplus_{n\ge1}\Z(2)
    \subseteq K=\prod_{n\ge1}\Z(2)
\]
the subspace topology $\tau$ inherited from $K$. Then $E$ is
a non-compact metrizable nuclear group and $\C(E)=\{\tau\}$.
\end{proposition}

\begin{proof}
The subgroup $E$ is proper and dense in $K$, so it is precompact.
If it were compact, it would be closed in the Hausdorff group
$K$, contradicting proper density. It is metrizable and nuclear.

Every continuous character of $E$ extends uniquely to its compact
completion $K$. A continuous character of $K$ depends on finitely
many coordinates. Consequently,
\[
  \widehat E\cong\bigoplus_{n\ge1}\Z(2)
\]
is countable, and the coordinate characters show that
$\tau=\sigma(E,\widehat E)$.

Let $\nu\in\C(E)$. A character of a group of exponent two takes
values in $\{0,1/2+\Z\}$, whose intersection with $\T_+$ is
$\{0\}$. Hence every quasi-convex subset of $E$ is an intersection
of character kernels and is a subgroup. The topology $\nu$
therefore has a zero-neighborhood base of open subgroups.

If a subgroup $U$ in this base had infinite index, $E/U$ would
be a countably infinite-dimensional $\F_2$-vector space. Its full
character group has cardinality $2^{\aleph_0}=\cont$. Pullback
along the continuous quotient map from $(E,\nu)$ to the discrete
group $E/U$ would inject this character group into
$(E,\nu)^\wedge=\widehat E$, contradicting countability.

Thus every subgroup $U$ in the base has finite index. The
characters of the finite quotient $E/U$ separate its points,
so $U$ is an intersection of finitely many kernels of
$\nu$-continuous characters. Compatibility makes these
characters $\tau$-continuous, whence $U$ is $\tau$-open.
It follows that $\nu\le\tau$. Conversely,
$\tau=\sigma(E,\widehat E)\le\nu$, so $\nu=\tau$.
\end{proof}

This example prevents each conclusion of Theorem D of
\cite{AD20} from extending to all non-compact nuclear groups:
neither its cardinality bound, its width bound, nor its chain
conclusion holds for $E$. The example also identifies the relevant
distinction: non-compact complete groups are non-precompact,
whereas a non-compact nuclear group may be precompact.

\subsection{Failure of the discrete order-isomorphism reduction}

\begin{definition}
We call the greatest element $t$ of a poset \emph{order-compact}
if, whenever a nonempty directed subset has supremum $t$, one
of its members equals $t$.
\end{definition}

This property depends only on the order and is therefore
preserved by order isomorphisms.

\begin{lemma}\label{lem:discrete-top-compact}
For every discrete abelian group $D$, the greatest element
$\delta$ of $\C(D)$ is order-compact.
\end{lemma}

\begin{proof}
The topology $\delta$ is the greatest element because it is
discrete. Every nonempty join of members of $\C(D)$ is
Hausdorff and locally quasi-convex, and lies between the Bohr
topology and $\delta$. It is therefore compatible, so joins of
topologies compute suprema in $\C(D)$.

Suppose a nonempty directed family $\{\mu_i:i\in I\}$ has
supremum $\delta$. Since $\{0\}$ is a neighborhood in the join,
there are finitely many indices $i_1,\ldots,i_r$ and
$\mu_{i_j}$-zero-neighborhoods $U_j$ such that
$U_1\cap\cdots\cap U_r\subseteq\{0\}$. Directedness provides
$i\in I$ with $\mu_{i_j}\le\mu_i$ for every $j$. Thus $\{0\}$
is a $\mu_i$-neighborhood and $\mu_i=\delta$. If $D$ is trivial,
the assertion follows immediately.
\end{proof}

\begin{proposition}\label{prop:no-discrete-reduction}
Let $G=\R^\N$ have its product topology. Then $G$ is complete
metrizable and nuclear, and
\[
  \C(G)\not\cong\C(D)
  \quad\text{for every discrete abelian group }D.
\]
\end{proposition}

\begin{proof}
Let $\tau$ be the product topology and
$\sigma=\sigma(G,\widehat G)$. The group $G$ is locally
quasi-convex, complete, and metrizable. Its original topology
$\tau$ is therefore the greatest member of $\C(G)$ by
\cite[Theorem 4.2]{CMT99}. For $n\ge1$, put
\[
  \nu_n=\sigma\vee\pi_{\le n}^{-1}(\tau_{\R^n}),
\]
where $\pi_{\le n}$ is the projection onto the first $n$
coordinates. Here the inverse-image topology is allowed to be
non-Hausdorff before taking the join. It is locally quasi-convex,
and $\sigma\le\nu_n\le\tau$, so $\nu_n\in\C(G)$.

On the $(n+1)$-st coordinate axis, the topology induced by
$\nu_n$ is the Bohr topology of $\R$. Indeed, the
finite-coordinate topology induces the indiscrete topology
on that axis, and the restrictions of the characters of $G$
give exactly the continuous characters of $\R$: each such
character extends by composition with the coordinate projection.
The Bohr topology is strictly coarser than the usual topology,
since the former is precompact and the latter is not.
Consequently, $\nu_n<\tau$ for every $n$.

The family $(\nu_n)$ is increasing. Its topological join is
$\tau$, because the usual topologies on finite sets of
coordinates generate the product topology. Hence its supremum
in $\C(G)$ is $\tau$. This shows that $\tau$ is not
order-compact, whereas the greatest element of every $\C(D)$
is order-compact by Lemma~\ref{lem:discrete-top-compact}.
The posets cannot be order-isomorphic. Finally, $G$ is nuclear
as a product of locally compact abelian groups.
\end{proof}

Thus the abstract discrete reduction in Theorem C of
\cite{AD20} fails in both classes mentioned in Question 6.8.
These examples delimit the extensions of Theorems C and D
to the classes in that question.

\appendix
\section{Sparse characters on the integers and the Pr\"ufer groups}
\label{app:sparse}

The construction below uses sparse arithmetic characters of the
type employed in \cite[Sections 4 and 5]{AD20}, together with the
almost disjoint coding in the proof of
Theorem~\ref{main:embedding}. It gives explicit order embeddings
for the two basic groups in Question 6.1 without using the
compact-polar construction.

Fix a prime $p$ and let $H=\Z$ or $H=\Z(p^\infty)$, with the
discrete topology. Put $S=\{2^n:n\ge1\}$ and choose an almost
disjoint family $\{A_\xi:\xi<\cont\}\subseteq[S]^\omega$.
For $a\in S$, define
\[
  \chi_a(x)=
  \begin{cases}
    x/p^a+\Z,&H=\Z,\\
    p^a x,&H=\Z(p^\infty)\subseteq\T.
  \end{cases}
\]
For $A\subseteq S$, put $K_A=\{\chi_a:a\in A\}$ and
$\tau_A=\tau_{K_A}$, the topology of uniform convergence on
$K_A$. For the estimates, abbreviate
$q_A=q_{K_A}$. Proposition~\ref{prop:uniform-basic} gives the
quasi-convex zero-neighborhood base
\[
  V(A,k)=V(K_A,k)=\{x:q_A(x)\le1/(4k)\}
    =\{j\chi_a:a\in A,\ 1\le j\le k\}\prepol.
\]
The auxiliary topology may be non-Hausdorff. With
$\sigma=\sigma(H,\Hom(H,\T))$, however, each
\[
  \nu_J=\sigma\vee\bigvee_{\xi\in J}\tau_{A_\xi},
  \qquad J\subseteq\cont,
\]
is Hausdorff, locally quasi-convex, and compatible with the
discrete topology of $H$, by
Proposition~\ref{prop:uniform-joins}(c).

Fix $\xi$, write $B=A_\xi$, and put $W_\xi=V(B,p)$. If
$\xi\in J$, $W_\xi$ is a $\nu_J$-zero-neighborhood.
Suppose $\xi\notin J$ and let $A$ be a finite union of sets
$A_\zeta$ with $\zeta\in J$. Then $B\setminus A$ is infinite.
For $b$ in this difference, put
\[
  x_b=
  \begin{cases}
    p^{b-1},&H=\Z,\\
    p^{-(b+1)}+\Z,&H=\Z(p^\infty).
  \end{cases}
\]
Distinct $a,b\in S$ satisfy either $a\le b/2$ or $a\ge2b$.
If $H=\Z$, $\chi_a(x_b)=0$ for $a<b$, while for $a>b$ its
norm is at most $p^{b-1-a}\le p^{-b-1}$. If
$H=\Z(p^\infty)$, $\chi_a(x_b)=0$ for $a>b$, while for
$a<b$ its norm is at most $p^{a-b-1}\le p^{-b/2-1}$.
Consequently,
\[
  q_A(x_b)\le p^{-b/2-1}\longrightarrow0
  \qquad(b\in B\setminus A).
\]
For $b<b'$ in $B\setminus A$, on the other hand,
\[
  \chi_b(x_b-x_{b'})=1/p+\Z\quad(H=\Z),
\]
and
\[
  \chi_{b'}(x_b-x_{b'})=-1/p+\Z
    \quad(H=\Z(p^\infty)).
\]
In either case, $q_B(x_b-x_{b'})\ge1/p>1/(4p)$.

Let $F\subseteq\Hom(H,\T)$ be finite and $\eta>0$.
Choose an infinite tail of $B\setminus A$ on which
$q_A(x_b)<\eta/2$. Total boundedness of $\T^F$ gives two
indices $b<b'$ in this tail such that
$\norm{\psi(x_b)-\psi(x_{b'})}<\eta$ for all $\psi\in F$.
Thus $x_b-x_{b'}$ belongs to the basic $\nu_J$-neighborhood
determined by $A,F,\eta$ but lies outside $W_\xi$. Hence
\[
  W_\xi\text{ is a }\nu_J\text{-zero-neighborhood}
  \quad\Longleftrightarrow\quad\xi\in J.
\]
As in the proof of Theorem~\ref{main:embedding}, this criterion
shows that $J\mapsto\nu_J$ is an order embedding. It provides
an antichain of cardinality $2^{\cont}$ for each of $\Z$ and
$\Z(p^\infty)$; the upper bound for topologies on a countable
set then gives equality of both cardinality and width with
$2^{\cont}$.

% Authors: supply funding and competing-interest declarations, as applicable.
% Do not state that there was no funding or no conflict without verification.

\section*{Declaration of generative AI and AI-assisted technologies
in the manuscript preparation process}
OpenAI Codex was used to assist with English-language editing,
manuscript organization, LaTeX preparation, and preliminary checks
of proofs and references. The authors are responsible for the
final mathematical content and wording of the manuscript.
% Before submission, the authors should review all changes and, if accurate,
% add: After using this tool, the authors reviewed and edited the content
% as needed and take full responsibility for the content of the article.

\end{document}